\documentclass[12pt]{amsart}
\usepackage[english]{babel}
\usepackage[colorinlistoftodos]{todonotes}
\usepackage[pdftex,paper=a4paper,portrait=true,textwidth=450pt,textheight=675pt,tmargin=3cm,marginratio=1:1]{geometry}

\usepackage{amsfonts}
\usepackage{amsmath}
\usepackage{bbm}
\usepackage{amsthm}
\usepackage{amssymb}
\usepackage{enumerate}
\usepackage{cancel}
\usepackage{color}
\usepackage[hidelinks]{hyperref}
\usepackage{tikz-cd}
\usepackage{enumitem}

\newtheorem{theorem}{Theorem}[section]
\newtheorem{corollary}[theorem]{Corollary}
\newtheorem{proposition-definition}[theorem]{Proposition-Definition}

\newtheorem{proposition}[theorem]{Proposition}
\newtheorem{lemma}[theorem]{Lemma}

\theoremstyle{definition}
\newtheorem{definition}[theorem]{Definition}
\newtheorem{remark}[theorem]{Remark}
\newtheorem{example}[theorem]{Example}

\newcommand{\mE}{\mathcal{E}}

\newcommand{\mF}{\mathcal{F}}

\newcommand{\mO}{\mathcal{O}}

\newcommand{\mZ}{\mathbb{Z}}
\newcommand{\mR}{\mathbb{R}}

\newcommand{\mC}{\mathbb{C}}
\newcommand{\mQ}{\mathbb{Q}}

\newcommand{\ds}{\displaystyle}

\newcommand{\cN}{\mathcal{N}}

\newcommand{\Hom}{\text{Hom}}
\newcommand{\Spec}{\text{Spec}\,}

\newcommand{\nocontentsline}[3]{}
\newcommand{\tocless}[2]{\bgroup\let\addcontentsline=\nocontentsline#1{#2}\egroup}

\begin{document}
	
	\title{Toric Foliated Minimal Model Program}
	\author{Weikun Wang}
	\address{Department of Mathematics, Southern University of Science and Technology, Shenzhen 518055, China}
	\email{\href{wangwk@sustech.edu.cn}{wangwk@sustech.edu.cn}}
	
	\begin{abstract}
		Using the theory of Klyachko filtrations for reflexive sheaves on toric varieties, we give a description of toric foliations and their singularities in terms of combinatorial data. We extend Spicer's results about co-rank one toric foliations to those of any ranks on $\mathbb{Q}$-factorial toric varieties and prove that the  toric foliated Minimal Model Program exists.  
	\end{abstract}
	\maketitle
	\tableofcontents

	\section{Introduction}

Over the last few decades, a lot of work has been done to understand the birational geometry of foliations. In particular, the canonical divisor of a foliation plays a similarly important role as the canonical divisor of a variety.  Therefore it is very natural to ask if we can analyze and classify foliations as what we do in the classical Minimal Model Program.

In the case of rank $1$ foliations on surfaces, McQuillan \cite{McQ08} and Brunella \cite{Bru15} gave a detailed classification of foliations in terms of the numerical Kodaira dimension. For threefolds, Spicer and Cascini proved in \cite{Spi20} and \cite{CS21} the existence of minimal models for F-dlt foliated pairs of co-rank $1$. Recently, in \cite{CS20}, they extended many of McQuillan's results to rank $1$ foliations on threefolds.

For higher dimensional varieties, the field still remains quite unexplored. But for toric foliations, it was shown by Spicer in \cite{Spi20} that the foliated Minimal Model Program can be run in the case of co-rank $1$. In this paper, we extend his results to toric foliations of any ranks on $\mQ$-factorial toric varieties.

To achieve this goal, we need a combinatorial description of torus equivariant sheaves on toric varieties. Introduced by Klyachko in \cite{Kly90} and then developed by Perling \cite{Per04a} and Kool \cite{Koo11}, the theory of Klyachko filtrations tells us how to describe a toric sheaf on a smooth toric variety in terms of multi-filtrations of vector spaces. Indeed, many of the results can be extended to $\mQ$-factorial toric varieties as in \cite{CT22} or even smooth toric Deligne-Mumford stacks as in \cite{GJK17} and \cite{Wan20}.

In the first section, we recall Klyachko's description of toric sheaves on toric varieties. Then we describe toric foliations in terms of multi-filtrations. In \cite{Pan15}, the first Chern class of a toric foliation is calculated directly. We get the same result by the combinatorial method given by \cite{Koo11} and \cite{CT22}.
\begin{theorem}
	\label{11}
	Let $X$ be a $\mQ$-factorial toric variety with its fan in the lattice $N$. For every subspace $V$ of $N_\mC$, there is an associated toric foliation $\mF_V$. The first Chern class of $\mF_V$ is $\sum_{\rho \subset V} D_\rho$, where $D_\rho$
	is the divisor corresponding to the ray $\rho$. Hence the canonical divisor of the foliation, defined as $K_{\mF_V} := -c_1(\mF_V)$, is $-\sum_{\rho \subset V} D_\rho$, where the sum is over all non $\mF_V$-invariant toric divisors. 
\end{theorem}

In the second section, we generalize the results of \cite{CLS11} and \cite{Mat13} for canonical divisors of toric varieties and show that singularities of toric foliations can be also described in terms of the geometry of convex cones. 
\begin{theorem} \label{12}
	Let $X$ be a $\mQ$-factorial toric variety with its fan $\Sigma$ in the lattice $N$. Let $\mF$ be the toric foliation associated to the subspace $V$ of $N_\mC$. Denote by $V_\mR$ the real vector space generated by real vectors of $V$. For every cone $\sigma$ in the fan $\Sigma$, set $S_1^\sigma =	\{m_{\sigma, V} \leq 1 \} \cap N \cap V_\mR  \cap \sigma $, $S_2^\sigma = \{m_{\sigma, V} \leq 0 \} \cap N \cap \{ N_\mR \setminus V_\mR\} \cap \sigma$ and $T_1^\sigma =	\{m_{\sigma, V} < 1 \} \cap N \cap V_\mR  \cap \sigma $. Then the criteria for $(X, \mF)$ to have terminal or canonical singularities are given as follows: 
	$$	\begin{array}{lll}
		\text{terminal} & \Longleftrightarrow & S_1^\sigma \cup S_2^\sigma = \{ 0 \} \cup \{v_i : \text{primitive vectors of rays in } \sigma \}  \quad   \forall \sigma \in \Sigma ,\\
		\text{canonical} & \Longleftrightarrow & T_1^\sigma = \{0\} \quad   \forall \sigma \in \Sigma .
	\end{array}$$ 
\end{theorem}

In the last section, we first prove the cone theorem for toric foliations of any rank.

\begin{theorem} \label{13}
	Let $\mF$ be a toric foliation on a $\mQ$-factorial toric variety $X$. Then $$\overline{NE}(X)_{K_\mF <0} = \sum \mR^+ [M_i],$$ where $M_i$ are torus invariant curves tangent to the foliation.
\end{theorem}

Similar to \cite{Spi20}, we are also able to prove the toric foliated MMP for any rank with results of \textbf{\hyperref[39]{Lemma \ref*{39}}} and \textbf{\hyperref[310]{Lemma \ref*{310}}}.

\begin{theorem} \label{14}
	Let $\mF$ be a toric foliation with canonical singularities on a $\mQ$-factorial toric variety $X$. Then the MMP for $\mF$ exists and ends with a foliation such that $K_\mF$ is nef or a fibration $\pi: X \to Z$ such that $\mF$ is pulled back from a foliation on $Z$. Furthermore, $\mF$ has non-dicritical singularities and the resulting foliation will also have non-dicritical singularities.
\end{theorem}

\tocless{\section*{Acknowledgement}}

The author would like to thank Zhan Li and Hang Zhao for their informative discussion about the theory of foliations during the algebraic geometry seminar at the Southern University of Science and Technology. The author also would like to thank the anonymous referee for his or her helpful comments. \\

\section{Toric Foliations}

\begin{definition}
	\cite[Definition 2.1]{AD13} A foliation on a normal variety $X$ is a nonzero coherent saturated subsheaf $\mF \subsetneq T_X$ that is closed under the Lie bracket.
	
	The rank of the foliation is the rank of the sheaf $r(\mF)$.
	
	A foliated variety is a pair $(X, \mF)$ such that $X$ is a normal variety and $\mF$ is a foliation on $X$.
\end{definition}

Since $T_X$ is reflexive, it implies that $\mF$ is also reflexive.

\begin{definition}
	\cite[Definition 2.5]{AD13} Let $\mF$ be a foliation on a normal variety $X$. Consider the Pfaff field of rank $r$ on $X$ given by:
	$$\eta: \Omega_X^r \to \wedge^r (T_X^*) \to \wedge^r (\mF^*) \to \det(\mF)^* = \mO_X(K_\mF). $$
	The singular locus of $\mF$, denoted by $\text{Sing}(\mF)$, is the closed subscheme of $X$ whose ideal sheaf is the image of the map $\Omega_X^r \otimes \mO_X(-K_\mF) \to \mO_X$ induced by $\eta$. 
\end{definition}

\begin{definition}\label{203}
	A subvariety $W \subset X$ is said to be tangent to the foliation $\mF$ if the tangent space of $W$ factors through $\mF$ along $X-\text{Sing}(X) \cup \text{Sing}(\mF)$.
	
	$W$ is said to be invariant if $\mF$ factors through the tangent space of $W$. 
\end{definition}

\begin{definition}
	Let $X$ be a toric variety. A foliation on $X$ is called toric if the sheaf $\mF$ is torus equivariant. 
\end{definition}

To describe a toric foliation, we first need to establish a relation between the saturated subsheaves of $T_X$ and the subspaces $V$ of $N_\mC$. 

There are two different approaches to establish this relation and obtain \textbf{\hyperref[11]{Theorem \ref*{11}}}. We first recall the simpler one given by Pang in \cite{Pan15}. Then we present the alternative one using Klyachko's filtrations, 

\begin{proposition}\cite[Proposition 2.1.2]{Pan15}
	Let $X$ be a $\mQ$-factorial toric variety with its fan $\Sigma$ in the lattice $N$. We can associated to each $n \in N$ a torus invariant derivation on $\mC[M]$ by 
	$$\delta_n: \mC[M] \to \mC[M], \quad \chi^m \to \langle m , n \rangle \cdot \chi^m.$$
	The morphism $\delta: N \to \text{Der}(\mC[M]), n \to \delta_n$ gives an isomorphism: $$N \otimes \mathcal{O}_X \cong T_X(-\log D),$$
	where $D = \sum_{\rho \in \Sigma(1)} D_\rho$.
	Indeed, the definition of $\delta_n$ can be extended to $N_\mC$.
\end{proposition}

\begin{proposition}\cite[2.1.4--2.1.6]{Pan15}
	For a given subspace $V$ of $N_\mC$, there is a unique torus equivariant subsheaf $V \otimes \mathcal{O}_X \subset T_X(-\log D) \subset T_X$. Let $\mF_V$ be the saturation of $V \otimes \mathcal{O}_X$ in $T_X$. This establishes an equivalence between the category of toric foliations and the category of subspaces of $N_\mC$.
\end{proposition}
Remark that it should be $N_\mC$ instead of $N_\mR$ in Pang's thesis. 

To calculate the first Chern class, we need to write down the local generators of $\mF_V$.

For $\rho \in \Sigma(1)$, we have
$$\rho^{\vee} \cap M  \cong \mZ_{\geq 0} m_\rho \oplus \mZ m_2 \oplus \cdots \oplus  \mZ m_n$$
where $m_2, \cdots, m_n$ are generators of $\rho^{\perp} \cap M$ and $ \langle m_\rho , v_\rho \rangle =1$.

If $X$ is a smooth toric variety and $\rho$ is contained in a maximal cone $\sigma$, then $m_i$ can be chosen as primitive vectors of rays in the dual cone $\sigma^\vee \cap M$. This is not true in general when $X$ is $\mQ$-factorial. But we can always choose $m_i \in \sigma^\vee \cap M$ so that $\delta_{v}|_{U_\rho}$ can be extended to $U_\sigma$.

Nevertheless, we have $U_\rho = \Spec  \mC[\rho^{\vee} \cap M] \cong \mC^* \times \mC^{n-1} $ and ${T_X|}_{U_\rho}$ is generated by $$\frac{\partial}{\partial{\chi^{m_\rho}}}, \chi^{m_2} \frac{\partial}{\partial{\chi^{m_2}}} , \cdots , \chi^{m_n} \frac{\partial}{\partial{\chi^{m_n}}}.$$

\begin{proposition} \cite[Lemma 2.1.10]{Pan15} \label{211}
	For all $v \in N_\mC$, set $$\delta_{v}|_{U_\rho}= \langle m_\rho, v \rangle \chi^{m_\rho}\frac{\partial}{\partial{\chi^{m_\rho}}} + \sum_{i=2}^n \langle m_i, v \rangle \chi^{m_i}\frac{\partial}{\partial{\chi^{m_i}}}. $$
	If $\rho \subset V$, we extend $v_\rho$ to a basis $\{v_\rho, v_2, \cdots, v_r  \}$ of $V$. If $\rho \not\subset V$, we choose one basis $\{v_1, v_2, \cdots, v_r  \}$ of $V$. Then $\mF_V|_{U_\rho}$ is generated by 
	$$	\left\{
	\begin{array}{ll}
		\ds \frac{1}{\chi^{m_\rho}}\delta_{v_\rho} , \delta_{v_2}, \cdots, \delta_{v_r} &  \text{ if } \rho \subset V,\\[1em]
		\delta_{v_1} , \delta_{v_2}, \cdots, \delta_{v_r} & \text{ if } \rho \not\subset V.
	\end{array} \right. $$
\end{proposition}

Then we obtain \textbf{\hyperref[11]{Theorem \ref*{11}}}, which is also \cite[Theorem 2.1.8]{Pan15}.

Notice that although the number of local generators of $\mF_V$ over $U_\rho$ is equal to the rank of $V$, it is not true over $U_\sigma$. 

\begin{example}
	Let $X$ be the toric surface given by the cone $\sigma$ in $\mZ^2$. Let $V=\text{Span} \langle (1,0) \rangle$ be the subspace of $\mC^2$.
	$$
	
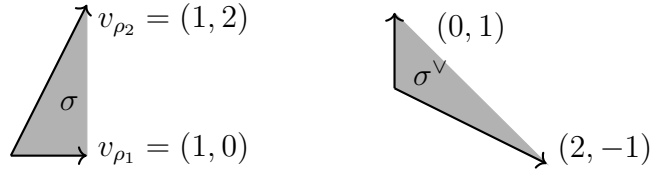
\begin{tikzpicture}
		\filldraw[black!30!white] (0,0) -- (1,0) -- (1,2);
		\draw[thick,->] (0,0) -- (1,0);
		\draw[thick,->] (0,0) -- (1,2);
		\node[right] at (1,0.1) {$v_{\rho_1} = (1,0)$};
		\node[right] at (1,1.8) {$v_{\rho_2} =(1,2)$};
		\node[right] at (0.5,0.7) {$\sigma$};
	\end{tikzpicture}
	\quad \quad \quad \quad
	\begin{tikzpicture}
		\filldraw[black!30!white] (0,0) -- (2,-1) -- (0,1);
		\draw[thick,->] (0,0) -- (2,-1);
		\draw[thick,->] (0,0) -- (0,1);
		\node[right] at (0.4,0.8) {$(0,1)$};
		\node[right] at (2,-0.8) {$(2,-1)$};
		\node[right] at (0.1,0.2) {$\sigma^\vee$};
	\end{tikzpicture}	
	$$
	
	Since $U_{\rho_2} \cong \text{Spec } \mC[x^2 y^{-1}, x^{-2} y, x]$, we can choose $m_{\rho_2}=(1,0)$ and $m_2=(2,-1)$. By the previous proposition,  $\mF_V|_{U_{\rho_2}}$ is generated by $x \frac{\partial}{\partial x} + 2 \chi^{(2,-1)}\frac{\partial}{\partial \chi^{(2,-1)}}$.
	Similarly, $U_{\rho_1} \cong \text{Spec } \mC[x, y, y^{-1}]$ and $\mF_V|_{U_{\rho_1}}$ is generated by $\frac{\partial}{\partial x}$.  
	
	Over $U_0 \cong T$, $x \frac{\partial}{\partial x} + 2 \chi^{(2,-1)}\frac{\partial}{\partial \chi^{(2,-1)}}$ can be simplified as $2x \frac{\partial}{\partial x}$. So $\mF_V|_T$ is indeed generated by one vector field. But over $X = U_\sigma \cong \text{Spec } \mC[y , x , x^2y^{-1}]$, $\mF_V$ cannot be generated by a single vector field.
\end{example}

Next, we give an alternative proof of \textbf{\hyperref[11]{Theorem \ref*{11}}} using Klyachko's filtrations, which is more general and involved. The reader can safely skip to the last part of this section about singularities.

To study moduli problems, Klyachko proved in \cite{Kly90} and \cite{Kly91} that torus equivariant torsion-free sheaves on toric varieties can be described by multi-filtrations of vector spaces that satisfy certain compatibility conditions. This method was further developed by Perling in \cite{Per04a} and \cite{Per04b} to classify arbitrary torus equivariant coherent sheaves. Here, we briefly recall this combinatorial description following the lines of \cite{Per03}.

\begin{remark}
	In the above-mentioned papers, the toric variety $X$ is required to be smooth. However, many of the results can be extended to $\mathbb{Q}$-factorial toric varieties without difficulties. Recall that a toric variety $X$ is $\mathbb{Q}$-factorial if the corresponding fan $\Sigma$ is simplicial, namely the primitive vectors of rays in every maximal cone are linearly independent. For more details, see \cite{KS98} and \cite{CT22}.
\end{remark}

Let $X$ be a toric variety given by a fan $\Sigma$ in the lattice $N$ of rank $n$. Let $M= \Hom(N,\mathbb{Z})$ be the character lattice. For each cone $\sigma \in \Sigma$, let $S_\sigma := \sigma^{\vee} \cap M$ and denote by $U_\sigma := \Spec k[S_\sigma]$ the corresponding affine toric variety. We can define a preorder $\leq_\sigma$ on $M$ by setting $m \leq_\sigma m'$ if and only if $m'-m \in S_\sigma$. A $\sigma$-family $\hat{F}
^\sigma$ is a collection of vector spaces $\{F^\sigma_m\}_{m \in M}$ with a collection of linear maps $\chi^\sigma_{m,m'}:  F^\sigma_m  \to  F^\sigma_{m'} $ for $m \leq_\sigma m'$ such that $\chi^\sigma_{m,m}=1$ and $\chi^\sigma_{m,m''}= \chi^\sigma_{m',m''} \circ \chi^\sigma_{m,m'}$ for $m \leq_\sigma m' \leq_\sigma m''$.

Let $\mF$ be a torus equivariant sheaf on a toric variety $X$. For each cone $\sigma \in \Sigma$, the module $F^\sigma := \Gamma (U_\sigma, \mF)$ has a natural torus action induced by the torus action on $\mF$. So it can be decomposed as a $M$-graded $k[S_\sigma]$-module $\ds F^\sigma = \bigoplus_{m\in M} F^\sigma_m$. 
It is proved in \cite[Theorem 4.5]{Per03} that the category of $\sigma$-families is equivalent to the category of torus equivariant quasi-coherent sheaves on $U_\sigma$.

Now consider a pair of cones $\tau \prec \sigma$ in the fan $\Sigma$. Denote by $i_{\tau, \sigma}: U_\tau \hookrightarrow U_\sigma$ the natural inclusion. Then the pullback module $i_{\tau , \sigma}^* (F^\sigma) = F^\sigma \otimes _{k[S_\sigma] }  k[S_\tau]$ can be decomposed as a $M$-graded $k[S_\tau]$-module. We call $\{\hat{F}^\sigma\}_{\sigma \in \Sigma}$ a $\Sigma$-family \cite[Definition 4.8]{Per03} if for each pair $\tau \prec \sigma$, there exists an isomorphism of families $\eta_{\tau,\sigma}:i_{\tau , \sigma}^* (\hat{F}^\sigma) \cong \hat{F}^\tau$ such that for each triple $\rho \prec \tau \prec \sigma$ there is an equality $\eta_{\rho,\sigma} = \eta_{\rho,\tau} \circ i_{\rho , \tau}^*\eta_{\tau,\sigma}$.

\begin{proposition} \cite[Theorem 4.9]{Per03}
	The category of $\Sigma$-families is equivalent to the category of torus equivariant quasi-coherent sheaves on a smooth toric variety $X$.
\end{proposition}

Since a foliation $\mF$ is reflexive, this combinatorial description can be further improved in the following way. Denote by $\sigma(1)$ the set of rays in the cone $\sigma$. For $\rho \in \sigma(1)$, denote by $v_\rho$ the primitive vector. Notice that $ \bigcup _{\rho \in \sigma(1)} U_\rho$ is an open subset of the affine toric variety $U_\sigma$ whose codimension is at least two. Hence $\Gamma(U_\sigma, \mF) = \Gamma(\bigcup _{\rho \in \sigma(1)} U_\rho, \mF)$. If we consider $\Gamma( U_\rho, \mF)$ as vector subspaces of $\Gamma(U_0, \mF)$, where $U_0$ is the torus $T$, then 
$$\Gamma(U_\sigma, \mF) \cong \bigcap_{\rho \in \sigma(1)} \Gamma( U_\rho, \mF). $$
Since $F^\rho := \Gamma (U_\rho, \mF)$ can be decomposed as $\ds F^\rho = \bigoplus_{m\in M} F^\rho_m$, we have $F^\sigma_m \cong \bigcap_{\rho \in \sigma(1)} F^\rho_m$.

Suppose the rank of $\mF$ is $r$. Since  
$\chi^\rho_{m,m'}:  F^\rho_m  \to  F^\rho_{m'}$ is an isomorphism if and only if $\langle m' -m , v_\rho \rangle = 0$, there is an identification of $M/\rho^{\perp}$ with $\mZ$ via the map $m \to \langle m , v_\rho \rangle $. We can then set the subspace $F^\rho (i) \subseteq \mC^r$ as $F^\rho_m$ wherever $\langle m , v_\rho \rangle = i$. Therefore we get an increasing filtration $\{\hat{F}^\rho\}$:
$$ 0 \subseteq \cdots \subseteq F^\rho (i) \subseteq F^\rho (i+1) \subseteq \cdots \subseteq \mC^r. $$

\begin{proposition} \cite[Theorem 4.21]{Per03}
	The category of multi-filtrations of vector spaces associated to each ray in $\Sigma(1)$ is equivalent to the category of torus equivariant reflexive sheaves on a $\mQ$-factorial toric variety $X$.
\end{proposition}

By \cite[Corollary 2.2.17]{DDK19}, the multi-filtrations associated to $T_X$ on a $\mQ$-factorial toric variety $X$ are given by
$$ {T}^\rho(i) = \left\{
\begin{array}{ll}
	0  &  i \leq - 2, \\
	\text{Span} \langle v_\rho \rangle & i=-1, \\
	N_\mC & i \geq 0.
\end{array}
\right.$$

Indeed, ${T_X|}_{U_\rho}$ is generated by 
$$\frac{\partial}{\partial{\chi^{m_\rho}}}, \chi^{m_2} \frac{\partial}{\partial{\chi^{m_2}}} , \cdots , \chi^{m_n} \frac{\partial}{\partial{\chi^{m_n}}}.$$
We can see that $\text{Span} \langle v_\rho \rangle$ corresponds to $\frac{\partial}{\partial{\chi^{m_\rho}}}$ as $ \langle -m_\rho , v_\rho \rangle =-1$ and $ \langle m_i , v_\rho \rangle = 0$.

Since foliations are saturated subsheaves of $T_X$, we can characterize them in terms of multi-filtrations.

\begin{proposition} \cite[Proposition 3.0.1]{DDK19}
	Let $\mE$ be an equivariant reflexive sheaf on a $\mQ$-factorial toric variety $X$ with multi-filtrations $\{E^\rho\}_{\rho \in \Sigma(1)}$ of the vector space $E$. Then there is a one-to-one correspondence between the equivariant saturated subsheaf $\mF$ of $\mE$ and multi-filtrations of the subspace $F \subseteq E$ with $F^\rho(i) = E^\rho(i) \cap F$. 
\end{proposition}

Hence for any given subspace $V \subseteq N_\mC$, we get a unique toric foliation $\mF_V$. 

\begin{proposition}
	The multi-filtrations corresponding to the foliation $\mF_V$ are given by 
	$$ {F_V^\rho}(i) = \left\{
	\begin{array}{ll}
		0  &  i \leq -2, \\
		\normalfont{\text{Span}} \langle v_\rho \rangle & i=-1, \\
		V & i \geq 0,
	\end{array} \right. \quad \text{ if } \rho \subset V,$$
	$$ {F_V^\rho}(i) = \left\{
	\begin{array}{ll}
		0  &  i \leq -1, \\
		V & i \geq 0,
	\end{array} \right. \quad \text{ if } \rho \not\subset V.$$
\end{proposition}

To determine the first Chern class of $\mF_V$, we need a combinatorial formula given by \cite{Koo11} for smooth toric varieties and by \cite{CT22} for $\mQ$-factorial toric varieties.

\begin{proposition} \cite[Corollary 2.18]{CT22} \label{29}
	Let $\mF$ be an equivariant reflexive sheaf on a $\mQ$-factorial toric variety $X$ with associated multi-filtrations $\{\hat{F}^\rho\}_{\rho \in \Sigma(1)}$. Then the first Chern class of $\mF$ is given by 
	$$c_1(\mF) = - \sum_{\rho \in \Sigma(1)} \sum_{i \in \mZ} i \dim F^{[\rho]}(i) D_\rho$$
	where $F^{[\rho]}(i) = F^\rho(i)/F^\rho(i-1)$ and $D_\rho$ is the Weil divisor corresponding to the ray $\rho$.
\end{proposition}

Now we obtain \textbf{\hyperref[11]{Theorem \ref*{11}}}.
\begin{proof}
	The calculation of the first Chern class can be easily seen from the previous two propositions. For the second part, notice that $\frac{\partial}{\partial{\chi^{m_\rho}}}$ is one of the generators of ${\mF_V|}_{U_\rho}$ if $\rho \subset V$.
\end{proof}

\begin{example}
	Consider the affine variety $\mC^3$ whose fan is given by $v_1=(1,0,0), v_2=(0,1,0), v_3=(0,0,1)$.
	
	Let $V_1$ be the vector space generated by $v_1+v_2, v_3$. Then $F_{V_1}$ is the foliation generated by $x \frac{\partial}{\partial x}+ y \frac{\partial}{\partial y}, \frac{\partial}{\partial z}$. The multi-filtrations are given by 
	$$ {F_V^{\rho_j}}(i) = \left\{
	\begin{array}{ll}
		0  &  i \leq -1, \\
		V_1 & i \geq 0,
	\end{array} \right. \quad j=1 \text{ or } 2 ,$$
	$$ {F_V^{\rho_3} }(i) = \left\{
	\begin{array}{ll}
		0  &  i \leq -2, \\
		\normalfont{\text{Span}} \langle v_3 \rangle & i=-1, \\
		V_1& i \geq 0.
	\end{array} \right. $$
	The canonical divisor is $K_{\mF_{V_1}}=-D_3$.
\end{example}

Next, we investigate toric foliations in terms of differential forms.

Given a foliation $\mF \subsetneq T_X$ of rank $r$, we set $\cN_\mF^* := (T_X/\mF)^*$ and $\cN_\mF :=(\cN_\mF^*)^*$, which are called the conormal sheaf and normal sheaf of $\mF$ respectively. Notice that $\cN_\mF^*$ is a reflexive subsheaf of $\Omega_X^1$ of rank $n-r$. Hence we can describe it in terms of Klyachko filtrations.

By \cite[Proposition 2.2.16]{DDK19}, the multi-filtrations associated to $\Omega_X^1$ on a $\mQ$-factorial toric variety $X$ is given by
$$ {\Omega}^\rho(i) = \left\{
\begin{array}{ll}
	0  &  i \leq -1, \\
	\text{Span} \langle v_\rho \rangle^\perp & i=0, \\
	M_\mC & i \geq 1.
\end{array}
\right.$$
Indeed, ${\Omega_X|}_{U_\rho}$ is generated by
$$\frac{ 1 }{\chi^{m_2}} d\chi^{m_2}, \cdots ,\frac{ 1 }{\chi^{m_n}}  d\chi^{m_n} ,  d{\chi^{m_\rho}}.$$
We can see that $\text{Span} \langle v_\rho \rangle ^\perp$ corresponds to $\frac{ 1 }{\chi^{m_2}} d\chi^{m_2}, \cdots ,\frac{ 1 }{\chi^{m_n}}  d\chi^{m_n}$.

Hence for any given subspace $W \subset M_\mC$, we get a unique conormal sheaf $\cN_{\mF_W}^*$.

\begin{proposition}
	The multi-filtrations corresponding to the conormal sheaf $\cN_{\mF_W}^*$ are given by 
	$$ (N_{\mF_W}^*)^\rho (i) = \left\{
	\begin{array}{ll}
		0  &  i \leq -1, \\
		\normalfont{\text{Span}} \langle v_\rho \rangle ^\perp \cap W & i= 0, \\
		W & i \geq 1,
	\end{array} \right. \quad \text{ if } W \not\subset \rho^\perp,$$
	$$ (N_{\mF_W}^*)^\rho (i) = \left\{
	\begin{array}{ll}
		0  &  i \leq -1, \\
		W & i \geq 0,
	\end{array} \right. \quad \text{ if } W \subset \rho^\perp.$$
\end{proposition}

\begin{theorem} \label{213}
	Let $W$ be a subspace of $M_\mC$. Then the first Chern class of the toric conormal sheaf $\cN_{\mF_W}^*$ is $-\sum_{W \not\subset \rho^\perp} D_\rho$. Hence the canonical divisor of the toric foliation $\mF_W$ is given by $K_{\mF_W} = -\sum_{W \subset \rho^\perp} D_\rho$.
\end{theorem}

\begin{proof}
	Notice that $\mO(K_X) = \mO(K_\mF) \otimes \det(\cN_\mF^*)$. The result then follows from \textbf{\hyperref[29]{Proposition \ref*{29}}}.
\end{proof}

Notice that $W \subset \rho^\perp$ if and only if $\rho \subset W^\perp$. Hence if $V^\perp=W$, we also obtain \textbf{\hyperref[11]{Theorem \ref*{11}}}. 

Next, we describe the singular locus of a toric foliation $\mF_V$.

\begin{theorem} \label{214}
	Let $\mF_V$ be a toric foliation on a $\mathbb{Q}$-factorial toric variety $X$ determined by a subspace $V$ of $N_\mC$. Consider the affine chart $U_\rho$ determined by the ray $\rho \in \Sigma(1)$. If $\rho \subset V$, we extend $v_\rho$ to a basis $\{v_\rho, v_2, \cdots, v_r  \}$ of $V$. If $\rho \not\subset V$, we choose one basis $\{v_1, v_2, \cdots, v_r  \}$ of $V$. Define the $r \times n$ matrix $A_{U_\rho}$ as 
	$$	\left\{
	\begin{array}{ll}
		$$\begin{bmatrix}
			1 & \cdots & 0  & \cdots & 0\\
			\langle m_\rho, v_2 \rangle \chi^{m_\rho} & \cdots & \langle m_i, v_2 \rangle \chi^{m_i} & \cdots & \langle m_n, v_2 \rangle \chi^{m_n}\\
			&\ddots & & & \vdots \\
			\langle m_\rho, v_r \rangle \chi^{m_\rho} & \cdots & \langle m_i, v_r \rangle \chi^{m_i} & \cdots & \langle m_n, v_r \rangle \chi^{m_n}\\
		\end{bmatrix}$$  &  \text{ if } \rho \subset V,\\[3em]
		$$\begin{bmatrix}
			\langle m_\rho, v_1 \rangle \chi^{m_\rho} & \cdots & \langle m_i, v_1 \rangle \chi^{m_i} & \cdots & \langle m_n, v_1 \rangle \chi^{m_n}\\
			&\ddots & & & \vdots \\
			\langle m_\rho, v_r \rangle \chi^{m_\rho} & \cdots & \langle m_i, v_r \rangle \chi^{m_i} & \cdots & \langle m_n, v_r \rangle \chi^{m_n}\\
		\end{bmatrix}$$  & \text{ if } \rho \not\subset V.
	\end{array} \right. $$
	Then the ideal sheaf of Sing($\mF_V$) on $U_\rho$ is determined by the ideal $I_\rho$ generated by all nonzero $r \times r$ minors of $A_{U_\rho}$.
\end{theorem}

\begin{proof}
	The result follows from \textbf{\hyperref[211]{Proposition \ref*{211}}} by calculating the $r$th wedge product of vector fields.
\end{proof}

\begin{corollary} \label{215}
	The singular locus of a toric foliation $\mF_V$ is empty if and only if for every maximal cone $\sigma \in \Sigma$, the vector space $V$ can be generated purely by rays in $\sigma$.
\end{corollary}    
\begin{proof}
	The result follows from the previous theorem.
\end{proof}

\section{Singularities}

It is shown in \cite[Proposition 14.3.1]{Mat13} that the criteria for a toric variety to have terminal or canonical singularities can be described in terms of the geometry of convex cones. In this section, we extend this result to singularities of toric foliations. 

First, we recall the definition of foliation singularities. Let $(X, \mF)$ be a foliated variety such that $K_\mF$ is $\mQ$-cartier. Given a projective birational morphism $\phi: X' \to X$, let $\mF'$ be the pulled-back foliation of $\mF$ on $X'$. We can write 
$$K_{\mF'} = \phi^* K_\mF + \sum a(E,X,\mF) E,$$
where $E$ runs over all exceptional divisors for $\phi$.

\begin{definition}
	The foliated variety $(X,\mF)$ is terminal (canonical) if for all $E$ exceptional over $X$, $a(E,X,\mF) > 0$ ($\geq 0$).
\end{definition}

Let $X$ be a $\mQ$-factorial toric variety. By \cite[Section 3.2]{CLS11}, for each cone $\sigma \in \Sigma$, there exists $m_\sigma \in M_\mQ$ such that $m_\sigma(v_i)=1$ for all primitive vectors $v_i$ of rays $\rho_i \in \sigma(1)$. Since $X$ is $\mQ$-factorial, $v_i$ are linearly independent and any vector $v_j$ contained in $\sigma$ can be written uniquely as $\sum_{\rho_i \in \sigma(1)} a_i v_i$ for $a_i \in \mQ$. Hence $$m_\sigma(v_j)=\sum_{\rho_i \in \sigma(1)} a_i m_\sigma(v_i).$$

There also exist $m'_\sigma, m_{\sigma, V} \in M_\mQ$ such that for primitive vectors $v_i$ of rays $\rho_i \in \sigma(1)$, we have
$$	m'_\sigma(v_i) = \left\{
\begin{array}{ll}
	0  &  \text{ if } v_i \in V, \\
	1 & \text{ if }  v_i \not\in V.
\end{array} \right. $$
$$m_{\sigma, V} (v_i) = \left\{
\begin{array}{ll}
	1  &  \text{ if } v_i \in  V, \\
	0 & \text{ if } v_i \not\in  V.
\end{array} \right.$$
Similarly, $m'_\sigma, m_{\sigma, V}$ are also well-defined for any vector $v_j$ contained in $\sigma$. Notice that $m_\sigma = m'_\sigma + m_{\sigma, V}$.

Now we obtain the \textbf{\hyperref[12]{Theorem \ref*{12}}}.

\begin{proof}
	Take a resolution of singularities $\phi: X' \to X$ by choosing a subdivision $\Sigma'$ of the fan $\Sigma$.
	
	Since  $\mO(K_X) = \mO(K_\mF) \otimes \det(\cN_\mF^*)$, we can rewrite $K_{\mF'} = \phi^* K_\mF + \sum a(E,X,\mF) E$ as
	$$K_{X'}- c_1(\cN_{\mF'}^*) = \phi^* (K_X - c_1(\cN_\mF^*)) + \sum a(E,X,\mF) E.$$
	
	Take a primitive vector $v_j \in N$ with $\langle v_j \rangle \in \Sigma'(1)$. Notice that the divisor $D_{\langle v_j \rangle}$ is exceptional with respect to $\phi$ if and only if $\langle v_j \rangle \not\in \Sigma(1)$.
	
	Suppose $v_j$ is contained in the cone $\sigma \in \Sigma$. Then
	$$\text{ord}_{D_{\langle v_j \rangle}} \{ {K_{X'} - \phi^* K_X} \} = -1 + m_\sigma(v_j).$$
	
	Since  $\mO(K_X) = \mO(K_\mF) \otimes \det(\cN_\mF^*)$, by \textbf{\hyperref[11]{Theorem \ref*{11}}}, $c_1(\cN_{\mF}^*) = -\sum_{\rho \not\subset V} D_\rho $. Then 
	
	$$
	\text{ord}_{D_{\langle v_j \rangle} } \{ c_1(\cN_{\mF'}^*) - \phi^* c_1(\cN_{\mF'}^*)  \}  = \left\{
	\begin{array}{ll}
		m'_\sigma (v_j)   &   \text{ if } v_j \in  V, \\
		- 1 + m'_\sigma (v_j) &   \text{ if } v_j \not\in V.
	\end{array} \right. 
	$$
	Hence the discrepancy at $D_{\langle v_j \rangle}$ is given by 
	$$
	a(D_{\langle v_j \rangle}, X, \mF) = \left\{
	\begin{array}{lll}
		-1 + m_\sigma(v_j) - m'_\sigma (v_j) & = -1 + m_{\sigma, V} (v_j)   &   \text{ if }  v_j \in V, \\
		m_\sigma(v_j) - m'_\sigma (v_j) & =  m_{\sigma, V} (v_j)  &   \text{ if }  v_j \not\in V.
	\end{array} \right.
	$$
	
	Therefore, $(X,\mF)$ is terminal at $U_\sigma$ if and only if for every vector $v_j$ contained in $\sigma$ such that $\langle v_j \rangle \notin \sigma(1)$, we have 
	$$
	\left\{
	\begin{array}{lll}
		-1 + m_{\sigma, V} (v_j) > 0   &   \text{ if } v_j \in V, \\
		m_{\sigma, V} (v_j) > 0 &   \text{ if } v_j  \not\in V.
	\end{array} \right. \\
	$$
	
	Denote by $V_\mR$ the real vector space generated by real vectors of $V$. 
	Let $S_1^\sigma =	\{m_{\sigma, V} \leq 1 \} \cap N \cap V_\mR  \cap \sigma $ and $S_2^\sigma = \{m_{\sigma, V} \leq 0 \} \cap N \cap \{ N_\mR \setminus V_\mR\} \cap \sigma$. Then $(X,\mF)$ is terminal at $U_\sigma$ if and only if
	$$S_1^\sigma \cup S_2^\sigma = \{ 0 \} \cup \{v_i : \text{primitive vectors of rays in } \sigma \}.$$
	Let $T_1^\sigma =	\{m_{\sigma, V} < 1 \} \cap N \cap V_\mR  \cap \sigma $ and $T_2^\sigma = \{m_{\sigma, V} < 0 \} \cap N \cap \{ N_\mR \setminus V_\mR \} \cap \sigma$. Then $(X,\mF)$ is canonical at $U_\sigma$ if and only if
	$$T_1^\sigma = \{0\}, T_2^\sigma = \emptyset.$$
	Notice that $m_{\sigma, V}$ is non-negative on $\sigma$. Hence $T_2^\sigma = \emptyset$ is always true.
\end{proof}

\section{Foliated MMP}

In this section, we extend Spicer's results \cite{Spi20} on co-rank $1$ toric foliations to those of any rank.

\begin{proposition}
	Let $\mF$ be a toric foliation on a $\mathbb{Q}$-factorial toric variety $X$, which is determined by a subspace $V$ of $N_\mC$. If $\rho \subset V$ for some $\rho \in \Sigma(1)$, then $\mF$ is pulled back along some dominant rational map $f: X \dasharrow Y$ with $\dim(Y) < \dim(X)$. In particular, if $K_\mF$ is not trivial, then $\mF$ is a pull-back.
\end{proposition}

\begin{proof}
	Let $V'$ be the subspace of $V$ generated by all the rays $\rho_i \subset V$. Consider the vector space given by $N'_\mR := N_\mR/ V'$. Then we can construct a new fan $\Sigma'$ determined by the projection of the fan $\Sigma$ onto the subspace $N'_\mR$. By proper choice of the new lattice $N'$ and adding necessary rays, we can assure that for every ray $\rho \in \Sigma(1)$, there exists a ray $\rho' \in \Sigma'(1)$ such that the projection of $\rho$ onto $N'$ is a multiple of $\rho'$. Denote the toric variety corresponding to $\Sigma'$ by $Y$. Then we get a toric map $f: X \dasharrow Y$.
	
	Notice that this map is only a toric morphism in codimension one. Indeed, it is possible that for some cone $\sigma \in \Sigma$, the projection of  $\sigma$ fails to fall into one cone in $\Sigma'$. Hence the projection mapping between two lattices is not compatible with the fans $\Sigma$ and $\Sigma'$ as defined in \cite[Section 3.3]{CLS11}. Hence $f$ is not a morphism in general. 
	
	Now suppose by contradiction that $\mF$ is not a pull-back. Then there exists no ray $\rho_i$ such that $\rho_i \subset V$. Hence $K_\mF$ is trivial by \textbf{\hyperref[11]{Theorem \ref*{11}}}.
\end{proof}

Before proving the next lemma, we need to generalize \textbf{\hyperref[203]{Definition \ref*{203}}} when the subvariety $W$ is contained in the singular locus.

\begin{definition}
	Let $\mF$ be a toric foliation on a $\mQ$-factorial toric variety $X$. We say that a curve $W \subset X$ is tangent to $\mF$ if, for any toric smooth resolution $\pi: X' \to X$, there exists a curve $W'$ on $X'$ that dominates $W$ and is tangent to $\mF'$ in the sense of \textbf{\hyperref[203]{Definition \ref*{203}}}, where $\mF'$ is the pulled back foliation on $X'$.
\end{definition}

\begin{lemma} \label{32}
	Let $\omega$ be a codimension $1$ cone in a simplicial fan $\Sigma$. Denote by $V_\omega$ the vector space generated by the rays of $\omega$. Then the torus invariant curve $D_\omega$ is not tangent to the foliation $\mF_V$ if and only if $V \subset V_\omega$.
\end{lemma}

\begin{proof}
	Consider the toric smooth resolution given by a sequence of subdivisions of the fan $\Sigma$. The cone $\omega$ is divided into multiple smooth codimension $1$ cones $\omega'_i$, whose corresponding curves $D_{\omega'_i}$ all dominate $D_\omega$. Since $\Sigma$ is simplicial, the vector space generated by the rays of $\omega'_i$ will be the same as $V_\omega$. Hence we can assume $\omega$ is contained in a maximal smooth cone $\sigma = \langle \omega , \rho \rangle$ for some ray $\rho$. By the change of coordinates, we can further assume $\omega= \langle e_1, \cdots, e_{n-1} \rangle$ and $\rho= \langle e_n \rangle$.
	
	Hence $D_\omega$ is not tangent to the foliation if and only if $\frac{\partial}{\partial{x_n}} \not\subset \mF_V|_{D_\omega}$. By \textbf{\hyperref[211]{Proposition \ref*{211}}}, this happens if and only if $\langle e_n, v \rangle = 0$ for any $v \in V$, which means  $V \subset V_\omega$.
\end{proof}

We recall here the byproduct of \textbf{\hyperref[11]{Theorem \ref*{11}}}.
\begin{lemma} \label{34}
	Let $\mF$ be a toric foliation on a $\mathbb{Q}$-factorial toric variety $X$, which is determined by a subspace $V$ of $N_\mC$. Let $D_\rho$ be the toric divisor corresponding to the ray $\rho$. Then $D_\rho$ is $\mF$-invariant if and only if $\rho \not\subset V$.
\end{lemma}

Now we are ready to prove the cone theorem for toric foliations.
\begin{proposition} \label{33}
	Let $\mF$ be a toric foliation on a $\mQ$-factorial toric variety $X$. Let $C$ be a curve in $X$ such that $K_\mF \cdot C <0$. Then $[C]=[M]+\alpha$ where $M$ is a torus invariant curve tangent to the foliation and $\alpha$ is a pseudo-effective class.
\end{proposition}

\begin{proof}
	By \cite[Chapter 14]{Mat13}, we have
	$$ [C]=\sum_{\text{tangent to }\mF} a_\zeta [D_\zeta] +  \sum_{\text{not tangent to }\mF} b_\omega [D_\omega]  $$
	where $D_\zeta$ and $D_\omega$ are torus invariant curves corresponding to codimension $1$ cones $\zeta$ and $\omega$.
	
	Next, we show that some $a_\zeta$ can be chosen as non-zero. Assume the contrary that $a_\zeta=0$ for all $\zeta$. 
	
	Since $D_\omega$ is not tangent to the foliation, by the previous lemma, we have $V \subset V_\omega$ where $V_\omega$ is the vector space generated by the rays of $\omega$. Suppose $\omega = \langle v_1, \cdots, v_{n-1} \rangle$, then at least one divisor $D_{v_i}$ is not foliation invariant.
	
	Since $K_\mF \cdot C <0$, we must have $D_\omega \cdot D_{\langle v_i \rangle} >0$ for some $\omega$, $v_i$. Let $v_n$, $v_{n+1}$ be two primitive vectors such that together with $\omega$ they generate two maximal cones $\tau_{n+1}$, $\tau_n$ of the fan $\Sigma$. Hence by \cite{Rei83} and \cite[Proposition 14.1.5]{Mat13}, we have 
	$\tau_i = \langle v_1, \cdots, \hat{v_i}, \cdots, v_{n-1}, v_n, v_{n+1} \rangle \in \Sigma$ and $\tau_n \cup \tau_{n+1}$ is strictly concave along $\iota_i= \langle v_1, \cdots, \hat{v_i}, \cdots, v_{n-1} \rangle $ as illustrated in the following figure.
	
	$$	
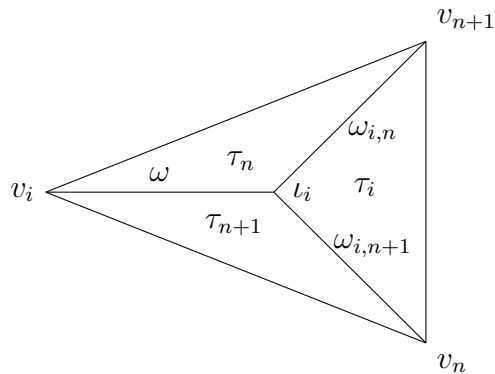
\begin{tikzpicture}
		\draw[thin,-] (-3,0) -- (0,0) ;
		\draw[] (-3,0) node[anchor=east] { $v_i$ };
		\draw[thin,-] (0,0) -- (2,2);
		\draw[] (2,2) node[anchor=south west] { $v_{n+1}$ };
		\draw[thin,-] (-3,0) -- (2,2);
		\draw[thin,-] (0,0) -- (2,-2);
		\draw[] (2,-2) node[anchor=north west] { $v_{n}$ };
		\draw[thin,-] (-3,0) -- (2,-2);
		\draw[thin,-] (2,2) -- (2,-2);
		\draw[] (-0.1,0.15) node[anchor=south east] { $\tau_n$ };
		\draw[] (0,-0.15) node[anchor=north east] { $\tau_{n+1}$ };
		\draw[] (-1.5,0) node[anchor=south] { $\omega$ };
		\draw[] (1.2,-0.2) node[anchor=south] { $\tau_i$ };
		\draw[] (1.3,-1) node[anchor=south] { $\omega_{i,n+1}$ };
		\draw[] (1.3,0.5) node[anchor=south] { $\omega_{i,n}$ };
		\draw[] (0.1,0) node[anchor=west] { $\iota_i$ };
	\end{tikzpicture}
	$$
	
	Suppose for example $D_{\omega_{i,n}}$ is not tangent to the foliation, then $V \subset V_{\omega_{i,n}}$ where $V_{\omega_{i,n} }$ is the vector space generated by rays in $\omega_{i,n}$. Hence $V \subset V_\omega \cap V_{\omega_{i,n}} = V_{\iota_i}$. But this implies ${\langle v_i \rangle} \not\subset V$, which means $D_{\langle v_i \rangle}$ is foliation invariant. Hence $D_{\omega_{i,n}}$ and $D_{\omega_{i,n+1}}$ are torus invariant curves tangent to the foliation. By \cite[Proposition 14.1.5]{Mat13}, $D_{\omega_{i,n}}$ and $D_{\omega_{i,n+1}}$ both belong to the same extremal ray as $D_\omega$.
\end{proof}

Now we obtain the \textbf{\hyperref[13]{Theorem \ref*{13}}}.

\begin{definition}
	Let $\mF$ be a toric foliation. It is said to have dicritical singularities if there exists a divisor $E$ sitting over the singular locus of $\mF$ which is not foliation invariant. Otherwise, it is called non-dicritical.
\end{definition}

By \textbf{\hyperref[214]{Theorem \ref*{214}}}, the singular locus of a toric foliation is always torus invariant. Hence every component is of the form $D_{\langle v_1, \cdots , v_m \rangle}$ for some cone $\tau =  \langle v_1, \cdots , v_m \rangle$. 

\begin{proposition} \label{37}
	Let $\mF$ be a toric foliation on a $\mathbb{Q}$-factorial toric variety $X$, which is determined by a subspace $V$ of $N_\mC$. Then $\mF$ is dicritical if there exists an singular component $D_{\langle v_1, \cdots , v_m \rangle}$ of the foliation such that $\text{dim} (\tau   \cap V ) \geq 1$, where $\tau =  \langle v_1, \cdots , v_m \rangle$, and $v_i \not\in V$ for $1 \leq i \leq m$.
\end{proposition}
\begin{proof}
	Every toric resolution of $D_{\langle v_1, \cdots , v_m \rangle}$ is given by successively adding new rays in the cone $\tau$ of the fan $\Sigma_X$, which is called star subdivision of $\Sigma$ along $\tau$ \cite{CLS11}. Hence a divisor sitting over $D_{\langle v_1, \cdots , v_m \rangle}$ is represented by a new ray in the cone $\tau$. This divisor is not foliation invariant if this ray is contained in $V$ by \textbf{\hyperref[34]{Lemma \ref*{34}}}, which means $\text{dim} (\{\tau \setminus \tau(1)\}  \cap V ) \geq 1$.
	
	Suppose $v_i \in V$, then by \textbf{\hyperref[214]{Theorem \ref*{214}}}, $D_{\langle v_i \rangle} \cap \text{Sing}(\mF) = \emptyset $. But $D_{\langle v_1, \cdots , v_m \rangle} \subset D_{\langle v_i \rangle}$. Hence $v_i \not\in V$ and $\text{dim} (\tau   \cap V ) \geq 1$.
\end{proof}

\begin{theorem} \label{38}
	Let $\mF$ be a toric foliation on a $\mathbb{Q}$-factorial toric variety $X$. If $\mF$ is canonical, then $\mF$ is non-dicritical.
\end{theorem}
\begin{proof}
	Suppose $\mF$ is dicritical. Then by the previous proposition, there exists a singular component $D_{\langle v_1, \cdots , v_m \rangle}$ of the foliation such that $\text{dim} (\tau   \cap V ) \geq 1$, where $\tau =  \langle v_1, \cdots , v_m \rangle$. Also, we have $v_i \not\in V$.
	
	Consider a vector $v \in \tau  \cap V$. Then $v= \sum_{\rho_i \in \tau(1)} a_i v_i$ and 
	$$m_{\sigma, V}(v)=\sum_{\rho_i \in \tau(1)} a_i m_{\sigma, V}(v_i) = 0$$ 
	as $v_i \not\in V$, where $\sigma$ is a maximal cone containing $\tau$. Hence $v \in T_1^\sigma$. By \textbf{\hyperref[12]{Theorem \ref*{12}}}, $\mF$ is not canonical.
\end{proof}

\begin{lemma} \label{39}
	Let $\mF$ be a toric foliation with canonical singularities on a $\mQ$-factorial toric variety $X$. Let $R$ be a $K_\mF$-negative extremal ray. Then there is a contraction corresponding to $R$ that belongs to one of the following types:
	\begin{enumerate}
		\item fibre type contractions,
		\item divisorial contractions,
		\item small contractions.
	\end{enumerate}
	By the previous theorem, $\mF$ has non-dicritical singularities. After a contraction of type ($1$) or ($2$), the resulting foliation will still have non-dicritical singularities.
\end{lemma}

\begin{proof} 
	The proof is similar to \cite[Lemma 10.8]{Spi20}. We only need to show the resulting foliation is non-dicritical.
	
	Suppose $\pi: X \to Y$ is a divisorial contraction. By the negativity lemma \cite[Lemma 3.39, Corollary 3.43]{KM98}, the resulting foliation $\mF_Y$ has canonical singularities. Hence $\mF_Y$ is non-dicritical by \textbf{\hyperref[38]{Theorem \ref*{38}}}.
	
	Suppose $\pi: X \to Y$ is a fibre type contraction. By \cite[Proposition 15.4.1 and 15.4.5]{CLS11}, the fan $\Sigma_Y$ in the lattice $N$ is a degenerate generalized fan \cite[Definition 6.22]{CLS11}. Then $\sigma_0 =  \cap_{\sigma \in \Sigma_Y} \sigma$ is the minimal cone in $\Sigma_Y$. Let $\overline{N} = N / (\sigma_0 \cap N)$ with quotient map $\phi: N \to \overline{N}$. Then $\overline{\Sigma}_Y=\{ \overline{\sigma} \mid  \sigma \in \Sigma_Y   \}$ is a usual fan in $\overline{N}$ such that $Y = Y_{\overline{\Sigma}_Y}$ and the foliation $\mF_Y$ is determined by $V_Y = V \cap \overline{N}_C$.
	
	Assume the contrary that $\mF_Y$ is dicritical. By \textbf{\hyperref[37]{Proposition \ref*{37}}}, there exists a singular point $D_{\tau} \in Y$ of the foliation $\mF_Y$ such that $\text{dim} (\tau   \cap V_Y ) \geq 1$, where $\tau =  \langle v_1, \cdots , v_m \rangle$ is a maximal cone in $\overline{N}$.
	
	By \cite[Proposition 15.4.5]{CLS11}, the fiber of $\pi$ at the point $D_{\tau}$ is a fake weighted projective space. Hence there exists a maximal cone $\tau_X = \langle \phi^{-1} (v_1), \cdots , \phi^{-1} (v_m), \cdots \rangle$ in $N_X = N$ that maps onto $\tau$. If the point $D_{\tau_X} \in X$ is not a singular point of the foliation $\mF$, then $V$ can be generated purely by rays in $\tau_X$ by \textbf{\hyperref[215]{Corollary \ref*{215}}}. But this is impossible since $V_Y$ cannot be generated purely by rays in $\tau$ as $D_\tau$ is singular. Hence the singular locus of $\mF$ is non-empty. 
	
	Since $\text{dim} (\tau_X   \cap V ) \geq \text{dim} (\tau   \cap V_Y ) \geq 1$, by \textbf{\hyperref[37]{Proposition \ref*{37}}}, $\mF$ is dicritical, which leads to a contradiction.

\end{proof}

\begin{lemma} \label{310}
	In the case of a small contraction, the flip exists and no infinite sequence of flips exists. If $\mF$ has canonical and hence non-dicritical singularities, then the flipped foliation $\mF^+$ does as well.
\end{lemma}

\begin{proof}
	The proof is similar to \cite[Lemma 10.9]{Spi20}. We only need to prove the non-dicriticality of $\mF^+$.
	
	Again, by the negativity lemma \cite[Lemma 3.39, Corollary 3.42]{KM98},  $\mF^+$ has canonical singularities. Hence $\mF^+$ is non-dicritical by \textbf{\hyperref[38]{Theorem \ref*{38}}}.
\end{proof}

\begin{example}
	$$	\begin{tikzpicture}
		\draw[thin,-] (-2,0) -- (2,0);
		\draw[thin,-] (2,0) -- (0,-2);
		\draw[thin,-] (2,0) -- (0,2);
		\draw[thin,-] (0,2) -- (-2,0);
		\draw[thin,-] (-2,0) -- (0,-2);
		\draw[] (-2,0) node[anchor=east] { $v_1$ };
		\draw[] (2,0) node[anchor=west] { $v_2$ };
		\draw[] (0,2) node[anchor=south] { $v_3$ };
		\draw[] (0,-2) node[anchor=north] { $v_4$ };
		\draw[dashed,->] (3,0) -- (5,0);
		\draw[thin,-] (-2,0) -- (0,0) ;
		\draw[thin,-] (8,2) -- (8,-2);
		\draw[thin,-] (10,0) -- (8,-2);
		\draw[thin,-] (10,0) -- (8,2);
		\draw[thin,-] (8,2) -- (6,0);
		\draw[thin,-] (6,0) -- (8,-2);
		\draw[] (6,0) node[anchor=east] { $v_1$ };
		\draw[] (10,0) node[anchor=west] { $v_2$ };
		\draw[] (8,2) node[anchor=south] { $v_3$ };
		\draw[] (8,-2) node[anchor=north] { $v_4$ };
	\end{tikzpicture}
	$$
	Let $v_1 = (1,0,0), v_2=(0,1,1), v_3=(0,1,0), v_4=(1,0,1)$. Let $\mF$ be the toric foliation determined by the vector space $V_{\langle v_3, v_4 \rangle }$. The linear relation is given by $-v_1-v_2+v_3+v_4=0$. The extremal rays of $X$ and $X^+$ are given by cones $\omega = \langle v_1, v_2 \rangle$ and $\omega^+ = \langle v_3, v_4 \rangle$ respectively.
	
	By the intersection formulas of toric variety \cite{CLS11}, we have 
	$K_\mF \cdot D_\omega = -2$ and $K_{\mF^+} \cdot D_{\omega^+} = 2$. Notice that $\mF$ is dicritical and $\mF^+$ has no singularities. Therefore $X \dasharrow X^+$ is a toric flip that turns a dicritical foliation into a non-dicritical one, but not the other way around.

\end{example}

Now we obtain the \textbf{\hyperref[14]{Theorem \ref*{14}}} with results of \textbf{\hyperref[39]{Lemma \ref*{39}}} and \textbf{\hyperref[310]{Lemma \ref*{310}}}.

\begin{proof} 
	The proof is similar to \cite[Lemma 10.10]{Spi20}.
\end{proof}




	
	
	

\end{document}